\documentclass[a4paper,10pt]{article}
\usepackage[utf8x]{inputenc}
\usepackage{graphicx,psfrag,xcolor}
\usepackage{amsmath,amssymb,amsthm}

\newcommand{\Z}{{\mathbb Z}}
\newcommand{\Cay}{{\rm Cay}}
\newcommand{\CAY}{\underline{{\rm Cay}}}
\newcommand{\Aut}{{\rm Aut}}
\newcommand{\Hom}{{\rm End}}

\definecolor{lgrey}{HTML}{C8C8C5}
\definecolor{dgrey}{HTML}{919494}

\newcommand{\lgrey}[1]{\textcolor{lgrey}{#1}}
\newcommand{\dgrey}[1]{\textcolor{dgrey}{#1}}

\newtheorem{thm}{Theorem}[section]

\newtheorem{prop}[thm]{Proposition}
\newtheorem{lem}[thm]{Lemma}

\newtheorem*{conj*}{Conjecture}

\theoremstyle{definition}
\newtheorem{rem}[thm]{Remark}
\newtheorem{quest}{Problem}

\title{On planar right groups}
\author{Kolja Knauer \and Ulrich Knauer}

\begin{document}

\maketitle

\begin{abstract}
In 1896 Heinrich Maschke characterized planar finite groups, that is
groups which admit a generating system such that the resulting
Cayley graph is planar. In our study we consider the question, which
finite \emph{semigroups} have a planar Cayley graph. Right
groups are a class of semigroups relatively close to groups. We
present a complete characterization of planar right groups.
\end{abstract}

\section{Introduction}
Cayley graphs of groups and semigroups enjoy a rich research history and are a classic point of interaction of Graph Theory and Algebra. While on the one hand knowledge about the group or semigroup gives information about the Cayley graph, on the other hand properties of the Cayley graph can be translated back to the algebraic objects and their attributes. A particular use of Cayley graphs is simply to \emph{visualize} a given group or semigroup by drawing the graph. It is therefore interesting to find Cayley graphs that can be drawn respecting certain aesthetic, geometric or topological criteria and then characterizing which groups or semigroups have Cayley graphs that can be drawn in such a way. Here we will focus on topological criteria, more precisely, on embeddability of the graph into orientable surfaces. The theoretic appeal of such type of question is that they now lie in the intersection of Graph Theory, Algebra, and Topology.

A group is called \emph{planar} if it admits a generating system
such that the resulting Cayley graph is planar, that is, it admits a
plane drawing. See Figure~\ref{fig:A42} for examples. Finite planar groups were characterized by
Maschke~\cite{Mas-96}, see Table~\ref{table}. There has been
considerable work towards a characterization of infinite planar
groups, see e.g.~\cite{Dro-06,Geo-11}.

 \begin{figure}[ht]
\begin{center}
  
%
  \psfrag{a}{~}
  \psfrag{b}{~}
  \psfrag{c}{~}
  \psfrag{d}{~}
  \psfrag{f}{~}
  \psfrag{g}{~}
  \psfrag{h}{~}
  \psfrag{i}{~}
  \psfrag{j}{~}
  \psfrag{k}{~}
  \psfrag{l}{~}
  \psfrag{e}{~}

     \includegraphics[width = \textwidth]{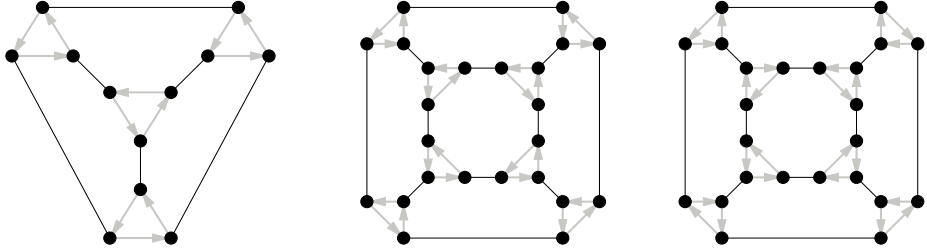}
    \caption{From left to right: the plane Cayley graphs $\Cay(A_4,\{(12)(34),\lgrey{(123)}\})$, $\Cay(S_4,\{\lgrey{(123)},(34)\})$, and $\Cay(\mathbb{Z}_2 \times A_4, \{ \lgrey{(0,(123))},(1,(12)(34)) \}) $.}\label{fig:A42}
\end{center}
 \end{figure}

With an analogous definition one might ask for planar semigroups.
Zhang studied planar Clifford semigroups~\cite{Zha-08}. Solomatin
characterized planar products of cyclic semigroups~\cite{Sol-06} and
described finite free commutative semigroups and some other types,
which are outerplanar~\cite{Sol-11}. For this he uses his own
planarity results, which are not easily accessible. We looked at toroidal right groups~\cite{Kna-10} but still no characterization
is known. In the present
paper we characterize planar right groups. Some of the
results have already been announced in~\cite{Kna-11}.

\bigskip

The paper is structured as follows:

  In Section~\ref{sec:pre} we introduce first notions of Cayley graphs of semigroups.
    In Section~\ref{sec:Maschke} we review Maschke's characterization of finite planar groups.
    In Section~\ref{sec:planar} we construct planar embeddings of those right groups which in the end will turn out to be exactly the planar ones.
    Section~\ref{sec:planarfromplanar} we study how Cayley graphs of the group are reflected in the Cayley graph of a right group. In particular, we reduce the set of right groups that have to be checked for planarity.
    In Section~\ref{sec:nonplanar} we prove that all candidates for planar right groups not shown to be planar before are not planar and therefore conclude the characterization.
    We close the paper with stating the characterization Theorem~\ref{thm:charact} and a set of open problems in Section~\ref{sec:con}.

\section{Preliminaries}\label{sec:pre}
Given a semigroup $S$ and a set $C\subseteq S$ we call the directed
multigraph $\Cay(S,C)=(V,A)$ with vertex set $V=S$ and directed arcs
$(s,sc)$ for all $s\in S$ and $c\in C$ the \emph{directed right
Cayley graph} or just the \emph{Cayley graph} of $S$ with
\emph{connection set} $C$. In drawings of Cayley graphs we will generally use different shades of gray for arcs corresponding to different elements of $C$. If a pair of vertices is connected by an arc in both directions we simply draw an undirected edge between them. See Figure~\ref{fig:A42} for an example.  

The \emph{genus} $\gamma(M)$ of an orientable surface $M$ is its number of handles. The \emph{genus} $\gamma(\Gamma)$ of a (directed) graph $\Gamma$ is the minimum genus over all surfaces it can be embedded in. The \emph{genus} $\gamma(S)$ of a semigroup $S$ is the minimum genus over all Cayley graphs $\Cay(S,C)$ where $C$ is a generating system of $S$.
We say that $S$ is \emph{planar} if $\gamma(S)=0$, i.e., there
exists a generating set $C$ of $S$ such that $\Cay(S,C)$ has a crossing-free drawing in the plane.

Clearly, when considering genus we may ignore edge-directions,
multiple edges and loops. We call the resulting simple undirected
graph the \emph{underlying graph} and denote it by $\CAY(S,C)$.

Given an undirected graph $\Gamma$ and a subset $E$ of the edges $E(\Gamma)$ of $\Gamma$, the \emph{deletion} of $E$ in $\Gamma$, denoted by $\Gamma\setminus E$, is simply the graph obtained from $\Gamma$ by suppressing all edges in $E$.
Further, we denote by $\Gamma/E$ the \emph{contraction} of $E$ in $\Gamma$. This is, the graph that arises from $\Gamma$ by identifying vertices that are adjacent via an edge in $E$ and deleting loops afterwards. A graph $\Gamma'$ is a \emph{minor} of $\Gamma$ if $\Gamma'$ can be obtained by a sequence of deletions and contractions from $\Gamma$. If $\Gamma'$ is obtained only using deletions or only contractions, we say that $\Gamma'$ is a \emph{deletion-minor} and \emph{contraction-minor}, respectively.  
Similarly, one defines the notions of deletion, contraction, and minors in directed graphs.

Given a graph $\Gamma$ we denote by $\Aut(\Gamma)$ its group of automorphisms and by $\Hom(\Gamma)$ its monoid of endomorphisms. 

The \emph{contraction lemma} due to Babai~\cite{Bab-73,Bab-77} is the following:
\begin{lem}~\label{lem:babai}
 Let $\Gamma$ be a connected graph and $G<\Aut(\Gamma)$. If the left action of $G$ on $\Gamma$ is fixpoint-free,
 then there exists $E\subset E(\Gamma)$ and a generating  set $C$ of $G$ such that $\CAY(G,C)\cong\Gamma/E$.
\end{lem}

A directed graph is called \emph{strongly connected} if for each pair $u,v$ there is a directed path from $u$ to $v$. Since for a minor $\Gamma'$ of $\Gamma$ one generally has $\gamma(\Gamma')\leq\gamma(\Gamma)$, we can apply Lemma~\ref{lem:babai} in order to obtain the following useful lemma:
\begin{lem}~\label{lem:semibabai}
 Let $S$ be a semigroup and $C\subseteq S$ such that $\Cay(S,C)$ is strongly connected.
 Then for every subgroup $G$ of $S$ we have $\gamma(G)\leq\gamma(\Cay(S,C))$.
\end{lem}
\begin{proof}
 For any $g\in G$ and $(s,sc)\in E(\Cay(S,C))$ by definition we have $(gs,gsc)\in E(\Cay(S,C))$.
 Hence, $g\in \Hom(\Cay(S,C))$. Clearly, $g^{-1}$ is the inverse mapping of $g$. Thus, $g\in \Aut(\Cay(S,C))$.
 To see that $G$ acts fixpoint-free, suppose $gs=s$ for some $s\in S$.
 Now for every out-neighbor $sc$ of $s$ we have $gsc=sc$, i.e., $sc$ is also a fixpoint of $g$.
 Choose a directed path $P$ from $s$ to some $h\in G$. We obtain $gh=h$, which implies $g=e$.
 Hence the left action of $G$ on $\Cay(S,C)$ is fixpoint-free.
 This implies that the same holds with respect to $\CAY(S,C)$.
 Thus, we may apply Lemma~\ref{lem:babai} and obtain some generating system,
 $C'$ of $G$ such that $\CAY(G,C')$ is a contraction-minor of $\CAY(S,C)$.
 Since contraction cannot increase the genus of a graph the claim follows.
\end{proof}

The above works also for the infinite case. Nevertheless, we are
only interested in the finite case. So, Lemma~\ref{lem:semibabai}
justifies to review the characterization of finite planar groups due
to Maschke~\cite{Mas-96}, which we will do in the next section.

\section{Maschke's Theorem -- a closer look}\label{sec:Maschke}

Before stating the theorem let us introduce some standard group notation:


%
%


$\Z_n=\{0,\dots,n-1\}$.  Note that $\Z_m\times \Z_n \cong \Z_{mn}$
if $gcd(m,n)=1$. In particular, $\Z_2\times \Z_2 \cong D_2$.

$D_n$ the dihedral group. The elements of $D_n$ are the symmetries
of the $n$-gon with the vertices $1,\dots,n$, that is $|D_n|=2n$. By
$\langle 13\rangle$  we denote the reflection around the axis
through 2, by $\langle 12\rangle$ the reflection around the axis
through the middle line between 1 and 2. The rotation by 1 is
denoted by $(1\dots n)$. Note that $\Z_2\times D_n \cong D_{2n}$, if
$n$ is odd.

$A_n, S_n$ the alternating group and the symmetric group,
respectively, on the $n$ points $1,\dots,n$, $n\le 5$. For their
elements we use the cycle notation.

The identity element is denoted by $e$ for all groups $G$ except for $\mathbb{Z}_n$, where we rather use $0$.

%
%
%
%
%
%
%
%
%
%
%
%
%
\begin{thm}[Maschke's Theorem]
 The groups and minimal generating systems in Table~\ref{table} are exactly those pairs having a planar Cayley graph.
 \end{thm}

\begin{table}
\begin{center}

 \begin{tabular}{c|c|c|c|c} 
 Group & Generators&$|V|$&$|E|$ &"Cayley" solid\\
 \hline \hline &&&& \\
$\Z_n$ & $1$&$n$&n  & $n$-gon
\\[.5ex]\hline &&&&
\\$\Z_2\times\Z_2$ & $(1,0),(0, 1)$&4&4&$2$-prism
\\$\Z_2\times\Z_{4}$ & $(1,0),(0,1)$&8&12&$\mathbf{cube}$
\\$ \Z_2\times \Z_n$&$(1,0),(0,1)$&$2n$ &$3n$& $n$-prism
\\[.5ex]\hline &&&&
\\$D_3$ & $(123),(12)$ &6 &9& 3-prism
\\$(\cong S_3)$ & $(123),(12),(23)$ & 6&12 &$\mathbf{octahedron}$
\\$D_4$ & $(1234),(13)$ &8 &12 &  $\mathbf{cube}$
\\$D_n$ & $\langle 12\rangle, \langle 13\rangle$ &$2n$& $2n$ & $2n$-gon
\\&$(1\dots n),\langle 12\rangle$ &&$3n$ & $n$-prism
\\[.5ex]\hline & & &&
\\$\Z_2\times D_{2n}$ &$(1,e),(0,\langle 12\rangle),(0,\langle 13\rangle)$&$4n$&$6n$&$2n$-prism
\\[.5ex]\hline && & &
\\ $A_4$& $(123),(12)(34)$&12&18&truncated tetrahedron
\\ & $(123),(234)$&&24&cuboctahedron
\\& $(123),(234),(13)(24)$&&30&$\mathbf{icosahedron}$
\\[.5ex]\hline&&&&
\\ $\Z_2\times A_4$ & $(0,(123)),(1,(12)(34))$&24&36& truncated cube
\\[.5ex]\hline & & &&
\\$S_4$ & $(123),(34)$&24&36&truncated cube
\\ & $(12),(23),(34)$&&36&truncated octahedron
\\ & $(12),(1234)$&&36&truncated octahedron
\\ & $(123),(1234)$&&48&rhombicuboctahedron
\\ & $(1234),(123),(34)$&&60&snub cuboctahedron
\\[.5ex]\hline & && &
\\$\Z_2\times S_4$ & $(1,(12))$,&48&72&(rhombi)truncated
\\& $(0,(23)),(0,(34))$&&&cuboctahedron
\\[.5ex]\hline&&& &
\\$A_5$&$(124),(23)(45)$ &60&90&truncated dodecahedron
\\&$(12345),(23)(45)$& &90&truncated icosahedron
\\&$(12345),(124)$ &&120&rhombicosidodecahedron
\\&$(12345),(124),(23)(45)$&&150 &snub icosidodecahedron
\\[.5ex]\hline& && &
\\$\Z_2\times A_5$&$(1,(12)(35))$,&120&180 &(rhombi)truncated
\\& $(1,(24)(35)),(1,(23)(45))$&& &icosidodecahedron

\end{tabular}
\end{center}
\caption{The planar groups, their minimal generating systems giving planar Cayley graphs, and the three-dimensional solids these graphs are the graphs of.}\label{table}
\end{table}



 \begin{figure}[ht]
\begin{center}
    \includegraphics[width = \textwidth]{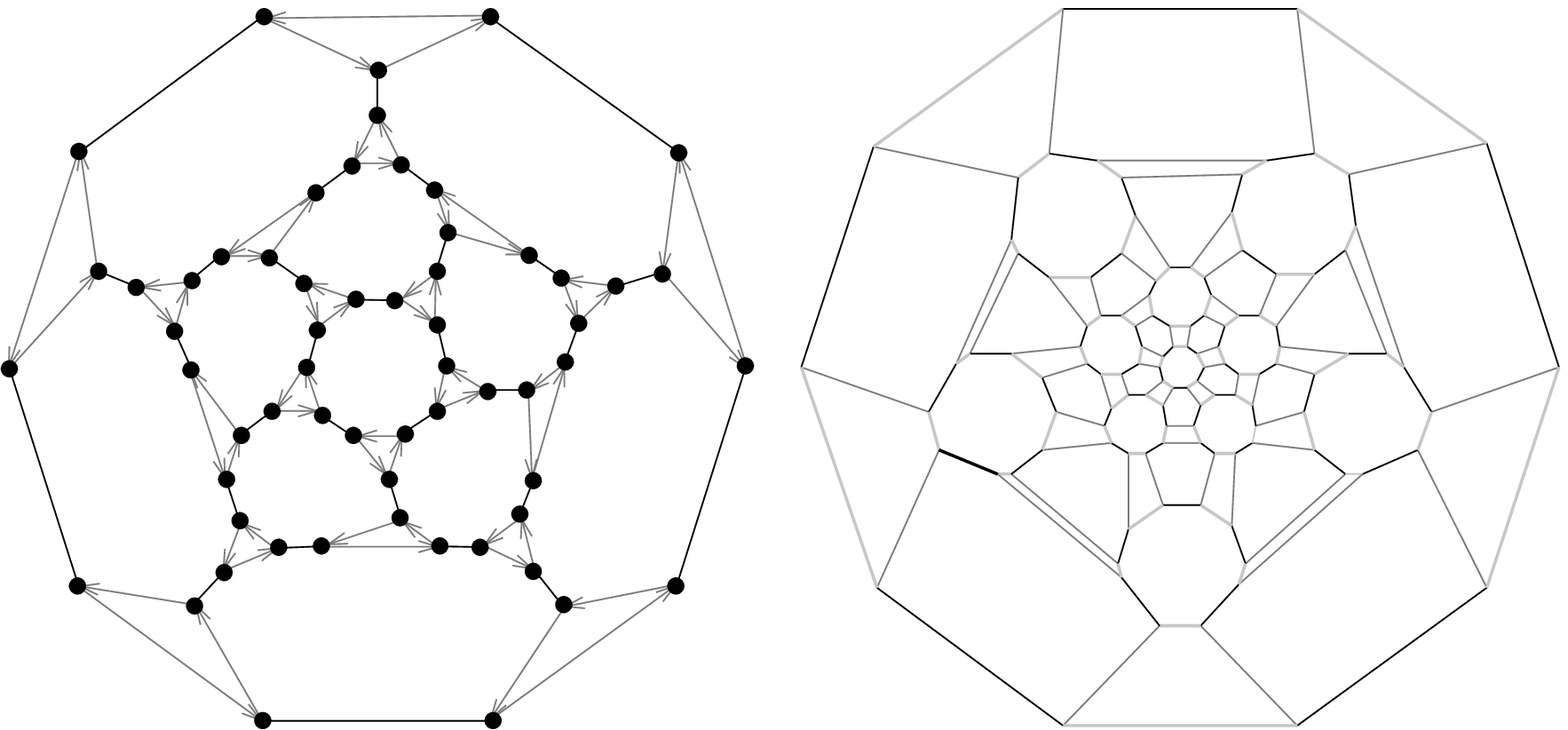}
    \caption{The two plane Cayley graphs $\Cay(A_5,\{\dgrey{(124)},(23)(45)\})$ and \newline
        $\Cay(\Z_2\times A_5,\{(1,(12)(35)),\dgrey{(1,(24)(35))},\lgrey{(1,(23)(45))}\})$.}\label{fig:A5}
\end{center}
 \end{figure}


\section{Planar right groups}\label{sec:planar}
The \emph{right zero band} on $k$ elements is the semigroup $R_k$ on
the set $\{r_1,\ldots ,r_k\}$ such that $r_ir_j:=r_j$ for all
$i,j\in[k]$, where through the entire paper we will denote $[k]:=\{1, \ldots, k\}$. A semigroup $S$ is called  a \emph{right group} if it
is isomorphic to the product $G\times R_k$ of a group and a right
zero band. See Figure~\ref{fig:CD} for examples.

On the way to characterize planar right groups we start
with positive results in this section.

\begin{figure}[ht]
\begin{center}
\psfrag{0,1}{\tiny{$(0,r_1)$}}
\psfrag{1,1}{\tiny{$(1,r_1)$}}
\psfrag{2,1}{\tiny{$(2,r_1)$}}
\psfrag{3,1}{\tiny{$(3,r_1)$}}
\psfrag{4,1}{\tiny{$(4,r_1)$}}
\psfrag{5,1}{\tiny{$(5,r_1)$}}
\psfrag{0,2}{\tiny{$(0,r_2)$}}
\psfrag{1,2}{\tiny{$(1,r_2)$}}
\psfrag{2,2}{\tiny{$(2,r_2)$}}
\psfrag{3,2}{\tiny{$(3,r_2)$}}
\psfrag{4,2}{\tiny{$(4,r_2)$}}
\psfrag{5,2}{\tiny{$(5,r_2)$}}
\psfrag{0,3}{\tiny{$(0,r_3)$}}
\psfrag{1,3}{\tiny{$(1,r_3)$}}
\psfrag{2,3}{\tiny{$(2,r_3)$}}
\psfrag{3,3}{\tiny{$(3,r_3)$}}
\psfrag{4,3}{\tiny{$(4,r_3)$}}
\psfrag{5,3}{\tiny{$(5,r_3)$}}
\psfrag{e,1}{\tiny{$(e,r_1)$}}
\psfrag{a,1}{\tiny{$(a,r_1)$}}
\psfrag{b,1}{\tiny{$(b,r_1)$}}
\psfrag{babab,1}{\tiny{$(a,r_1)$}}
\psfrag{baba,1}{\tiny{$(ab,r_1)$}}
\psfrag{bab,1}{\tiny{$(bab,r_1)$}}
\psfrag{ba,1}{\tiny{$(ba,r_1)$}}
\psfrag{e,2}{\tiny{$(e,r_2)$}}
\psfrag{a,2}{\tiny{$(a,r_2)$}}
\psfrag{b,2}{\tiny{$(b,r_2)$}}
\psfrag{babab,2}{\tiny{$(a,r_2)$}}
\psfrag{baba,2}{\tiny{$(ab,r_2)$}}
\psfrag{bab,2}{\tiny{$(bab,r_2)$}}
\psfrag{ba,2}{\tiny{$(ba,r_2)$}}
\psfrag{e,3}{\tiny{$(e,r_3)$}}
\psfrag{a,3}{\tiny{$(a,r_3)$}}
\psfrag{b,3}{\tiny{$(b,r_3)$}}
\psfrag{babab,3}{\tiny{$(a,r_3)$}}
\psfrag{baba,3}{\tiny{$(ab,r_3)$}}
\psfrag{bab,3}{\tiny{$(bab,r_3)$}}
\psfrag{ba,3}{\tiny{$(ba,r_3)$}}

\includegraphics[width = \textwidth]{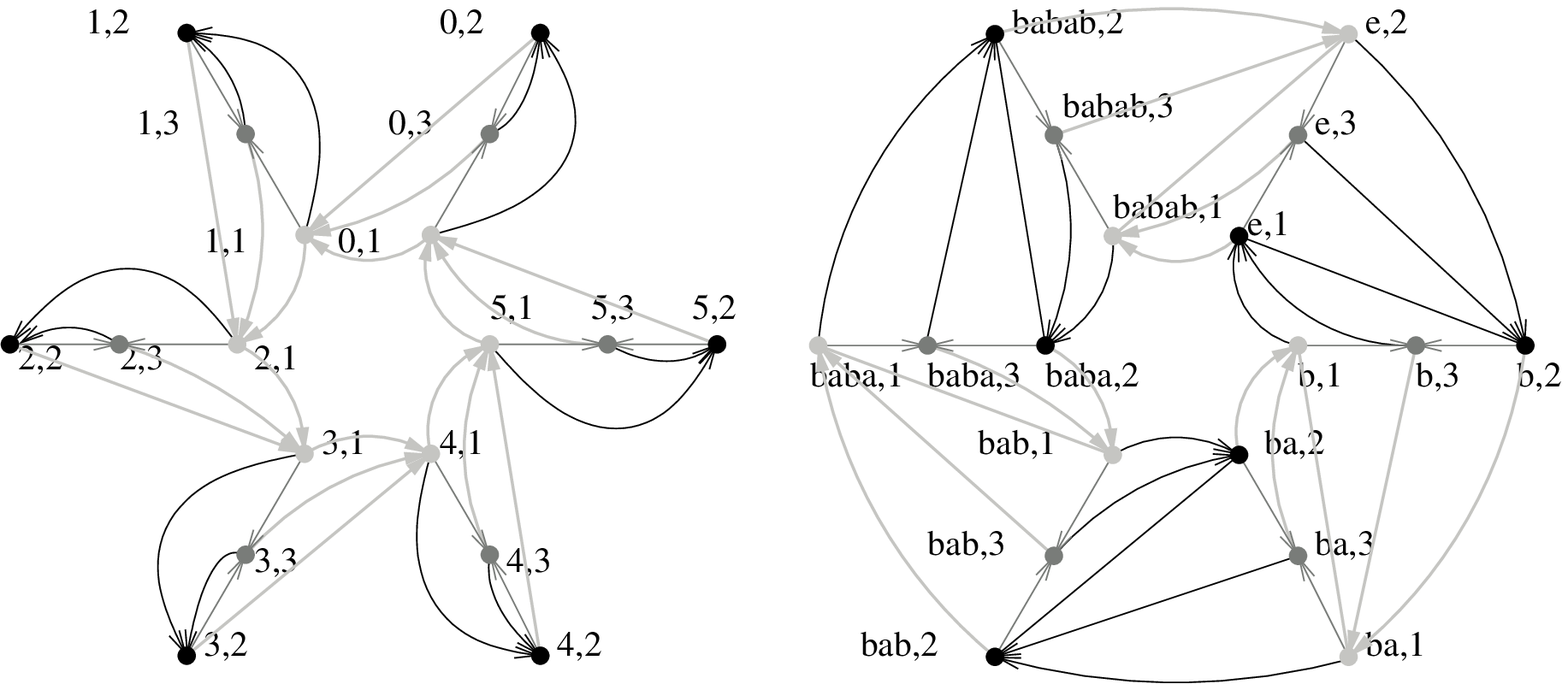}
\caption{Plane Cayley graphs $\Cay(\Z_6\times R_3,\{\lgrey{(1,r_1)},(0,r_2),\dgrey{(0,r_3)}\})$ and $\Cay(D_3\times
R_3,\{\lgrey{(a,r_1)},(b,r_2),\dgrey{(e,r_3)}\})\}$. Vertex colors stand for $\lgrey{G\times \{r_1\}}$, $G\times \{r_2\}$, and $\dgrey{G\times \{r_3\}}$, respectively.}\label{fig:CD}
\end{center}
\end{figure}

\begin{rem}\label{rem:leftgroup}{\rm Analogously one considers a \emph{left zero band}
$L_k=\{l_1,\dots,l_k\}$ on $k$ elements such that $l_il_j:=l_i$ for
all $i,j\in [k]$. Now consider the \emph{left group} $L_k\times G$
whose generating systems always have the form $L_k\times C$ where
$C$ is a generating system of $G$. The right Cayley graph
$\CAY(L_k\times G,L_k\times C)$ consists of $k$ copies of $\CAY(G,C)$.
 Consequently, a left group $L_k\times G$ is planar if and only if
 the group $G$ is planar, for arbitrary $k\in\mathbb{N}$.}
 \end{rem}

\begin{lem}\label{lem:ZD}
If $G\in\{\mathbb{Z}_n, D_n\}$ then $G\times R_2$ and $G\times R_3$ are planar.
\end{lem}
\begin{proof}
 If $G=\mathbb{Z}_n$ with generator $b$ then $\CAY(\mathbb{Z}_n,\{b\})$ is a cycle.
 Now $D:=\{(b,r_1), (e,r_2), (e,r_3)\}$ is a generating system of $\mathbb{Z}_n\times R_3$.
 A plane drawing of $\Cay(\mathbb{Z}_n\times R_3,D)$ is on the left of Figure~\ref{fig:CD}
 for the case $n=6$.

 If $G=D_n$ with two degree two generators $a,b$  then again $\CAY(D_n,\{a,b\})$
 is a cycle and $D:=\{(a,r_1), (b,r_2), (e,r_3)\}$ is a generating system of $D_n\times R_3$.
 A plane drawing of $\Cay(D_n\times R_3,D)$ is shown on the right of Figure~\ref{fig:CD} for the case $n=3$.
 
 Now, for $G\in\{\mathbb{Z}_n, D_n\}$ and $S':=G\times R_2$ note that $D':=D\backslash \{(e,r_3)\}$ is a generating system of $S'$ and $\Cay(S',D')$ is
 a subgraph of $\Cay(S,D)$. Thus, it is also planar.
\end{proof}

\begin{lem}\label{lem:local}
 Let $G$ be a group with generating system $C=\{a,b\}$ with $a^2=e$ and $b^2\neq e$.
 If $\Cay(G,C)$ has an embedding on a surface $M$ such that for each $a$-edge
  the $b$-arcs incident with it alternate in direction when ordered by the local
  rotation systems, see the left of Figure~\ref{fig:local1},
  then $\gamma(G\times R_2), \gamma(G\times R_3)\leq\gamma(M)$.
\end{lem}

\begin{figure}[ht]
\begin{center}

\psfrag{e}{~}
\psfrag{b}{~}
\psfrag{a}{~}
\psfrag{ab}{~}
\psfrag{ab-1}{~}
\psfrag{b-1}{~}
\psfrag{e,1}{~}
\psfrag{b,1}{~}
\psfrag{a,1}{~}
\psfrag{ab,1}{~}
\psfrag{ab-1,1}{~}
\psfrag{b-1,1}{~}
\psfrag{e,2}{~}
\psfrag{b,2}{~}
\psfrag{a,2}{~}
\psfrag{ab,2}{~}
\psfrag{ab-1,2}{~}
\psfrag{b-1,2}{~}
\psfrag{e,3}{~}
\psfrag{b,3}{~}
\psfrag{a,3}{~}
\psfrag{ab,3}{~}
\psfrag{ab-1,3}{~}
\psfrag{b-1,3}{~}

\includegraphics[width = .7\textwidth]{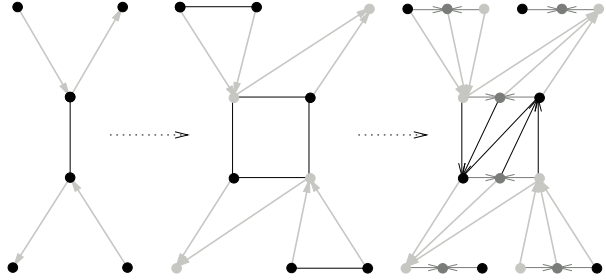}
\caption{Left: Local
configuration of $b$-arcs (gray) around an $a$-edge (black). Middle: Topologically relevant part
of the transformation. Right: Completion to graph of right group.
Arc and vertex colors stand for elements of $G\times \{r_1\}$, $\lgrey{G\times \{r_2\}}$, and $\dgrey{G\times \{r_3\}}$, respectively.}\label{fig:local1}
\end{center}
\end{figure}

\begin{proof}
 We first show the statement for $S:=G\times R_3$. As a generating system we use $D:=\{(a,r_1), (b,r_2), (e,r_3)\}$.
 The property of our embedding of $\Cay(G,C)$ on $M$ makes it
 possible to blow up
 $a$-edges to rectangles such that incoming
 $b$-arcs are attached to one pair of opposite vertices and outgoing $b$-arcs to an entire side of the rectangle each, see the middle of Figure~\ref{fig:local1}.
  This blown up graph can be completed to the wanted embedding of $\Cay(S,D)$. This
 is shown on
 the right of Figure~\ref{fig:local1}.

 For $S':=G\times R_2$ note that $D':=D\backslash \{(e,r_3)\}$ is a generating system of $S'$ and $\Cay(S',D')$
 is a subgraph of $\Cay(S,D)$. Thus, $\Cay(S',D')$ embeds in $M$. An example for this construction in the case $S':=G\times R_2$ and $G=A_4$ is depicted in Figure~\ref{fig:A4R22}.

 However, similarly one can see that choosing $D''=\{(e,r_1),(a,r_2),(b,r_2)\}$, also yields a graph $\Cay(S',D''')$
 which can be embedded into $M$. This constructions for the case $S':=G\times R_2$ is exemplified with $G=A_4$ in Figure~\ref{fig:A4R21}.
\end{proof}

\begin{figure}[ht]
\begin{center}
\psfrag{e,1}{\tiny{$(e,r_1)$}}
\psfrag{123,1}{\tiny{$((123),r_1)$}}
\psfrag{132,1}{\tiny{$((132),r_1)$}}
\psfrag{134,1}{\tiny{$((134),r_1)$}}
\psfrag{142,1}{\tiny{$((142),r_1)$}}
\psfrag{143,1}{\tiny{$((143),r_1)$}}
\psfrag{241,1}{\tiny{$((241),r_1)$}}
\psfrag{342,1}{\tiny{$((342),r_1)$}}
\psfrag{324,1}{\tiny{$((324),r_1)$}}
\psfrag{1234,1}{\tiny{$((12)(34),r_1)$}}
\psfrag{1324,1}{\tiny{$((13)(24),r_1)$}}
\psfrag{1432,1}{\tiny{$((14)(32),r_1)$}}
\psfrag{e,2}{\tiny{$(e,r_2)$}}
\psfrag{123,2}{\tiny{$((123),r_2)$}}
\psfrag{132,2}{\tiny{$((132),r_2)$}}
\psfrag{134,2}{\tiny{$((134),r_2)$}}
\psfrag{142,2}{\tiny{$((142),r_2)$}}
\psfrag{143,2}{\tiny{$((143),r_2)$}}
\psfrag{241,2}{\tiny{$((241),r_2)$}}
\psfrag{342,2}{\tiny{$((342),r_2)$}}
\psfrag{324,2}{\tiny{$((324),r_2)$}}
\psfrag{1234,2}{\tiny{$((12)(34),r_2)$}}
\psfrag{1324,2}{\tiny{$((13)(24),r_2)$}}
\psfrag{1432,2}{\tiny{$((14)(32),r_2)$}}

\includegraphics[width = .8\textwidth]{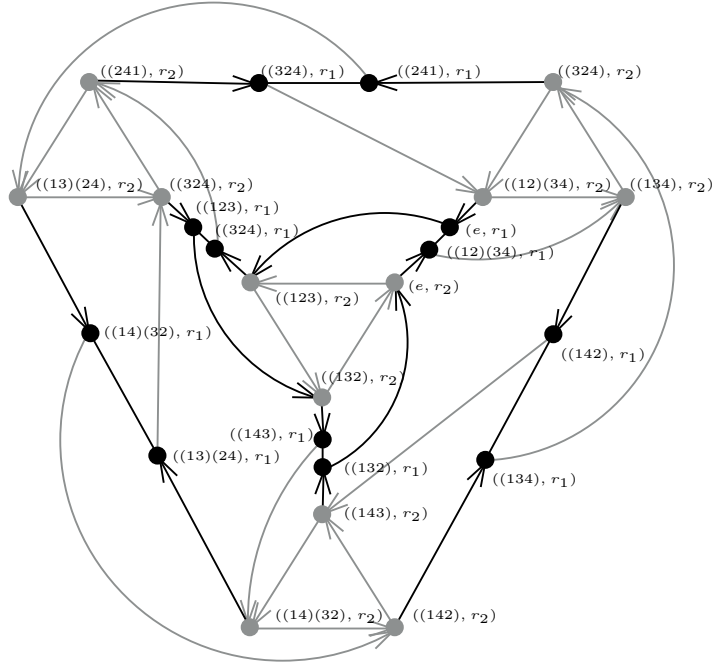}
\caption{The plane Cayley graph $\Cay(A_4\times R_2,\{((12)(34),r_1),\dgrey{(123),r_2)})\})$. This exemplifies the principal construction in Lemma~\ref{lem:local} and generalizes to $A_4\times R_3$. Vertex colors stand for elements of $G\times \{r_1\}$ and $\dgrey{G\times \{r_2\}}$, respectively.}\label{fig:A4R22}
\end{center}
\end{figure}

\begin{figure}[ht]
\begin{center}
\psfrag{e,1}{\tiny{$(e,r_1)$}}
\psfrag{123,1}{\tiny{$((123),r_1)$}}
\psfrag{132,1}{\tiny{$((132),r_1)$}}
\psfrag{134,1}{\tiny{$((134),r_1)$}}
\psfrag{142,1}{\tiny{$((142),r_1)$}}
\psfrag{143,1}{\tiny{$((143),r_1)$}}
\psfrag{241,1}{\tiny{$((241),r_1)$}}
\psfrag{342,1}{\tiny{$((342),r_1)$}}
\psfrag{324,1}{\tiny{$((324),r_1)$}}
\psfrag{1234,1}{\tiny{$((12)(34),r_1)$}}
\psfrag{1324,1}{\tiny{$((13)(24),r_1)$}}
\psfrag{1432,1}{\tiny{$((14)(32),r_1)$}}
\psfrag{e,2}{\tiny{$(e,r_2)$}}
\psfrag{123,2}{\tiny{$((123),r_2)$}}
\psfrag{132,2}{\tiny{$((132),r_2)$}}
\psfrag{134,2}{\tiny{$((134),r_2)$}}
\psfrag{142,2}{\tiny{$((142),r_2)$}}
\psfrag{143,2}{\tiny{$((143),r_2)$}}
\psfrag{241,2}{\tiny{$((241),r_2)$}}
\psfrag{342,2}{\tiny{$((342),r_2)$}}
\psfrag{324,2}{\tiny{$((324),r_2)$}}
\psfrag{1234,2}{\tiny{$((12)(34),r_2)$}}
\psfrag{1324,2}{\tiny{$((13)(24),r_2)$}}
\psfrag{1432,2}{\tiny{$((14)(32),r_2)$}}

\includegraphics[width = .8\textwidth]{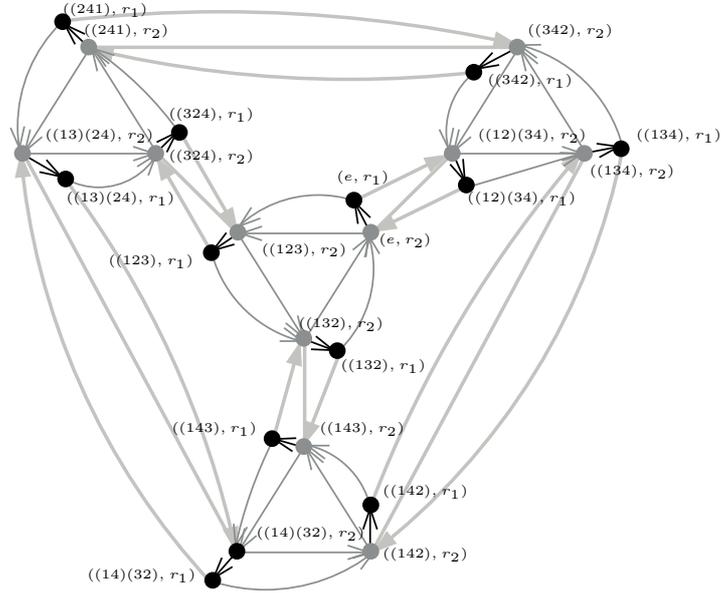}
\caption{The graph $\Cay(A_4\times R_2,\{(e,r_1),\lgrey{((12)(34),r_2)},\dgrey{(123),r_2)})\})$. This exemplifies the alternative construction for $G\times R_2$ mentioned in the proof of Lemma~\ref{lem:local}. Vertex colors stand for elements of $G\times \{r_1\}$ and $\dgrey{G\times \{r_2\}}$, respectively.}\label{fig:A4R21}
\end{center}
\end{figure}

\begin{thm}\label{thm:positiv}
 If $G\in\{\{e\},\mathbb{Z}_n, D_n, S_4, A_4, A_5\}$ then $G\times R_k$
and $\{e\}\times R_{k'}$ are planar for $k\leq 3$ and $k'\leq 4$.
\end{thm}
\begin{proof}
 For $\mathbb{Z}_n$ and $D_n$ this was proved in Lemma~\ref{lem:ZD}.
 For the remaining groups this follows by their plane drawings provided in Figure~\ref{fig:A42} and
 Figure~\ref{fig:A5} and applying Lemma~\ref{lem:local}. Moreover, $\CAY(\{e\}\times R_{k'}, R_{k'})$ is isomorphic to the
 complete graph on $k'$ vertices $K_{k'}$.
 Together this yields the claim.
\end{proof}

It remains to show that the list of planar right groups in
Theorem~\ref{thm:positiv} is complete.

\section{Planar right groups come from planar groups}\label{sec:planarfromplanar}
For $C\subseteq G\times R_k$ we denote the projections of $C$ on the
respective factors by $\pi_G(C):=\{g\in G\mid \exists j\in
[k]:(g,r_j)\in C\}$ and $\pi_{R_k}(C):=\{r_j\in R_k\mid \exists g\in
G:(g,r_j)\in C\}$. We start with the following basic lemma, which
does not hold in general products of groups and semigroups:

\begin{lem}\label{lem:strong}
Let $C\subseteq S=G\times R_k$, then the following are equivalent \begin{itemize}
\item[(i)] $C$ generates $S$,
\item[(ii)] $\pi_G(C)$ generates $G$ and $\pi_{R_k}(C)$ generates $R_k$,
\item[(iii)] $\Cay(S,C)$ is strongly connected.
\end{itemize}
\end{lem}
\begin{proof}
Clearly, (i) implies (ii). Now, we show (ii)$\Rightarrow$(iii): Take
$s=(g,r_i),t=(h,r_j)\in S$. Multiply $s$ from the right by a
sequence of elements of $C$ in order to obtain $(h,r_{\ell})$ for some
$\ell\in[n]$. Now, take some $(f,r_j)\in C$ and multiply it order of
$f$ many times to $(h,r_{\ell})$ from the right. We obtain a directed
path from $s$ to $t$. This yields that $\Cay(S,C)$ is strongly
connected.

To see, (iii)$\Rightarrow$(i), let $s\in S$. Any directed path in $\Cay(S,C)$ from a vertex $c\in C$ to $s$ corresponds
to a word of elements of $C$ generating $s$. By strong connectivity, such a path exists.
\end{proof}

Lemma~\ref{lem:semibabai} and Lemma~\ref{lem:strong} together yield
that in our characterization we only need to consider planar groups
as factors. Unfortunately, Lemma~\ref{lem:semibabai} does not
preserve the generating system of the right group we start out from.
Specializing to right groups we can obtain a stronger statement
under certain circumstances. For this we introduce the notation $\pi_G(C)_j:=\{g\in
G\mid (g,r_j)\in C\}$.

\begin{lem}\label{lem:factors}
 Let $S=G\times R_k$ and $C$ a generating set such that there exists $(g,r_j)\in C$ with $g^{-1}hg=h^{\pm 1}$ for all
 $h\in \pi_G(C)_j$ or for all $h\in \pi_G(C)\backslash\pi_G(C)_j$. Then $\CAY(G,\pi_G(C))$ is a minor of $\CAY(S,C)$.
\end{lem}
\begin{proof}
 Let $(g,r_j)\in C$ be the element claimed in the statement of the lemma. Delete all arcs of the form
 $(g',r_i),(g'f,r_{\ell})$ with $f\neq g$ and $i\neq\ell$. Contract all arcs of the form $((g',r_i),(g'g,r_j))$ for
 $i\neq j$. Since every vertex $(g',r_i)$ with $i\neq j$ has an arc of that form, we can view the new set of vertices as $G\times\{r_j\}$. In the contracted graph there is an arc $((g',r_j),(f,r_j))$ if it was there before, i.e., there is $(h,r_j)\in C$ with $g'h=f$, or if it arose from the contraction. In this case for some $i\neq j$ there are $((g'g^{-1},r_i),(fg^{-1},r_i))$ and $(h,r_i)\in C$ with $g'g^{-1}h=fg^{-1}$. This is, the new graph is isomorphic to $\Cay(G, \pi_G(C)_j\cup g^{-1}(\pi_G(C)\backslash\pi_G(C)_j)g)$. 
 
 Now we use the assumption on $(g,r_j)$. If we have $g^{-1}hg=h^{\pm 1}$ for all
 $h\in \pi_G(C)_j$, then $\Cay(G, \pi_G(C)_j\cup g^{-1}(\pi_G(C)\backslash\pi_G(C)_j)g)=\Cay(G,(g^{-1}\pi_G(C)g)^{\pm 1})$. But since 
 $g'h=f\Leftrightarrow g^{-1}g'gg^{-1}hg=g^{-1}fg$ we have that the contracted graph $\Cay(G,(g^{-1}\pi_G(C)g)^{\pm 1})$ is isomorphic to $\Cay(G,\pi_G(C)^{\pm 1})$. This means that $\CAY(G,\pi_G(C))$ is a minor of $\CAY(S,C)$.
 
 If we have $g^{-1}hg=h^{\pm 1}$ for all $h\in \pi_G(C)\backslash\pi_G(C)_j$, then $\Cay(G, \pi_G(C)_j\cup g^{-1}(\pi_G(C)\backslash\pi_G(C)_j)g)=\Cay(G,\pi_G(C)^{\pm 1})$. Hence, $\CAY(G,\pi_G(C))$ is a minor of $\CAY(S,C)$.
 \end{proof}

Many right groups satisfy the preconditions for
Lemma~\ref{lem:factors}. Note for instance that if the group
generated by some $\pi_G(C)_j$ is Abelian, then the preconditions of
Lemma~\ref{lem:factors} are trivially satisfied. A Coxeter system is
a group $G$ with generating system $C$, such that all generators are
of order $2$ and all relations are of the form
$(c_ic_j)^{m_{ij}}=e$. The \emph{Coxeter-Dynkin diagram} of a
Coxeter system is an (edge labeled) graph, on the vertex set $C$,
with an edge connecting $c_i$ and $c_j$ if and only if $m_{ij}\geq
3$. A consequence of the classification of finite Coxeter groups,
due to Coxeter~\cite{Cox-35} is that the Coxeter-Dynkin diagram of a
Coxeter system is a tree.

\begin{prop}
 Let $S=G\times R_k$ and $C$ a generating set. If $(G,\pi_G(C))$ forms a Coxeter system, then $\CAY(G,\pi_G(C))$ is a
 minor of $\CAY(S,C)$.
\end{prop}
\begin{proof}
 Let $g\in\pi_G(C)_j$ be a leaf of the corresponding Coxeter-Dynkin diagram, and $h$ its neighbor.
 If the only $i\in[k]$ with $h\in\pi_G(C)_i$ is $i=j$, then $(gf)^2=e$ for all $f\in\pi_G(C)\setminus\pi_G(C)_j$
 and Lemma~\ref{lem:factors} gives the result. Otherwise we can consider $C':=C\setminus\{(h,r_j)\}$,
 which satisfies Lemma~\ref{lem:strong}(ii) and therefore is still generating. Moreover, $\CAY(G,\pi_G(C'))$ is
 a deletion minor of $\CAY(G,\pi_G(C))$. We have that $(gf)^2=e$ for all $f\in\pi_G(C')_j$ and Lemma~\ref{lem:factors}
 gives the result.
\end{proof}

We believe, that a stronger form of Lemma~\ref{lem:factors} should hold.
\begin{conj*}
 Let $S=G\times R_k$ and $C$ a generating set. Then $\CAY(G,\pi_G(C))$ is a minor of $\CAY(S,C)$.
\end{conj*}

Note that for products $G\times S$ of group $G$ and semigroup $S$ not being a right zero band the natural generalization
of the conjecture is false. Indeed, even if $S=\mathbb{Z}_2$ it is false: Consider the planar representation of
$A_5\times \mathbb{Z}_2$
with three generators of order two, see Figure~\ref{fig:A5}. Since there is no planar representation of $A_5$ with three
generators of order two, no Cayley graph of $A_5$ arising by projecting the generating system of $A_5\times \mathbb{Z}_2$
to $A_5$ is a minor of the planar Cayley graph of $A_5\times \mathbb{Z}_2$.

We will prove a weakening of this conjecture for the planar case in Theorem~\ref{thm:project}. Essentially we will prove,
that in order to not satisfy the preconditions of Lemma~\ref{lem:factors} the Cayley graph must have too many edges
to be planar. Therefore, we need the classic result of Euler:

\begin{lem}[Euler's Formula]\label{lem:Euler}
Every simple connected plane graph $G$ with vertex set $V$, edge set $E$ and face set $F$ fulfills $|V|-|E|+|F|=2$.
In particular, $G$ has at most $3|V|-6$ edges and at most $2|V|-4$ edges if the embedding has no triangular faces.
\end{lem}

As a second ingredient we need a formula for the number of edges of the
underlying undirected Cayley graph of a right group. For $C\subseteq
G\times R_k$ and $a\in G$ we set $c_a:=|\{j\in [k]\mid (a,r_j)\in
C\}|$. Furthermore set $m:=|G|$.

\begin{lem}\label{lem:count}
Let $S=G\times R_k$ with generating system $C$. Then $\CAY(S,C)$ has $km$ vertices and its number of edges is:
\begin{eqnarray}\label{form:count}
 m((\sum_{a\in\pi_G(C)}c_ak-\frac{c_{a^{-1}}}{2})-\frac{c_e}{2})\geq \frac{m}{2}((2k-1)\sum_{a\in\pi_G(C)}c_a-c_e).
\end{eqnarray}

\end{lem}
\begin{proof}
The number of vertices is trivial.

 We start by proving the equality for the number of edges. Every element $(a,r_i)\in C$ contributes an outgoing arc at every element of $S$. But if $(a^{-1},r_j)\in C$ all arcs
 of the form $((g,r_j),(ga,r_i))$ are counted twice and there are $m$ of them. Note that this occurs in particular if
 $a^2=e$ and also if $i=j$. So, this yields $mkc_a-\frac{c_{a^{-1}}}{2}m$ edges labeled $a$. In the particular case
 that $a=e$ additionally at each vertex $(g,r_i)$ a loop can be deleted, i.e., instead of counting half an edge at each
 such vertex we count none. This yields the $-m\frac{c_e}{2}$ in the formula.

 Together we obtain the claimed equality.
 
 Observe now that for fixed $c_e$ the left-hand-side of the formula is minimized if $a^2=e$ for every $a\in\pi_G(C)$, i.e., $c_a=c_{a^{-1}}$. This yields the lower bound. 
\end{proof}

\begin{thm}\label{thm:project}
Let $S=G\times R_k$ and $C$ a generating system such that
$\CAY(S,C)$ is planar. Then $\CAY(G,\pi_G(C))$ is a minor of $\CAY(S,C)$, i.e., in particular planar.
\end{thm}
\begin{proof}
 The statement is trivial for $k=1$ so assume $k\geq 2$.
 If we cannot apply Lemma~\ref{lem:factors}
 we know in particular that $|\pi_G(C)_j|>1$ for all $j\in [k]$ and in particular $\sum_{a\in\pi_G(C)}c_a\geq 2k$.
 Moreover, we have $e\notin \pi_G(C)$. Thus $c_e=0$.

 We can now use Lemma~\ref{lem:count} to estimate the number of edges of $\CAY(S,C)$. Indeed, since $c_e=0$ we get the following lower bound from~\eqref{form:count}:
 $$\frac{m}{2}(2k-1)\sum_{a\in\pi_G(C)}c_a\geq (2k-1)mk>3mk-6, \text{ for } k\geq 2,$$
 whereas the smallest value in this chain is the upper bound for the number of edges of a planar graph given
 by Lemma \ref{lem:Euler} -- a contradiction.
\end{proof}

\section{Non-planar right groups from planar groups}\label{sec:nonplanar}

In this section we show that the right groups $\Z_2\times H\times
R_k$ with $H$ any of $\Z_{2(n+1)},D_{2n},A_4,S_4,A_5$ where $n\geq 1$ and $k\geq 2$ are not planar. Moreover, $\{e\}\times R_k$ is not planar for $k\geq 5$. Since
$\mathbb{Z}_2\times\mathbb{Z}_2\cong D_4$,  $\mathbb{Z}_2\times\mathbb{Z}_{2n+1}\cong \mathbb{Z}_{4n+2}$ and
$\mathbb{Z}_2\times D_{2n+1}\cong D_{4n+2}$ for all $n\geq 1$ this is exactly the set of right groups we have to prove to
be non-planar in order to show that the list from Theorem~\ref{thm:positiv} is complete.

Euler's Formula already allows us to restrict the size of the right zero band in the product of a planar right group:

\begin{prop}\label{prop:n<4}
 If $G$ is nontrivial and $G\times R_k$ is planar, then $k\leq 3$. Moreover, $G\times R_k$ is non-planar for any $G$ and $k\geq 5$.
\end{prop}
\begin{proof}
 Let $k\geq 4$ and $G\times R_k$ with $G$ non-trivial, i.e., there is $a\in G$ such that $c_a>0$. The lower bound in~\eqref{form:count} is minimized if $k=4$ and there is exactly one such $a\in G$ and $c_a=1$. In this case~\eqref{form:count} gives $m((3k-\frac{3}{2})+(k-\frac{1}{2})-\frac{3}{2})=12.5m>12m-6$ --
 a contradiction to Euler's Formula (Lemma~\ref{lem:Euler}).

 Since $\CAY(\{e\}\times R_k, R_k)\cong K_k$ and $K_k$ is non-planar for all $k\geq 5$ we obtain the second part of
 the statement.
\end{proof}

%

With some further edge counting we obtain:

\begin{prop}\label{prop:nonplanar1}
 The right groups $\Z_2\times H\times
R_k$ with $H$ any of $D_{2n},S_4,A_5$, where $n\geq 1$ and $k=2,3$, are
not planar.
\end{prop}
\begin{proof}
 Let $S=G\times R_k$ be one of the right groups from the statement and suppose that $C$ is a generating system of $S$
 such that $\CAY(S,C)$ is planar. By Theorem~\ref{thm:project} we know that there is a generating system
$C'\subseteq \pi_G(C)$ of $G$ such that $\CAY(G,C')$ is planar. Comparing with Table~\ref{table} we see that all planar
generating systems for our choice of $G$ consist of three generators all having order two, say $a_1,a_2,a_3$.

If (up to relabeling of $R_k$) we have $C'':=\{(a_1,r_1),(a_2,r_2),(a_3,r_3)\}\subseteq C$ (in particular $k=3$) then we consider the subgraph
$\CAY(S,C'')$ of $\CAY(S,C)$. By Lemma~\ref{lem:count} we know that $\CAY(S,C'')$ has $7.5m$ edges, where $m=|G|$.
Since $\pi_G(C'')$ contains only order $2$ elements and $\pi_G(C'')$ is a minimal generating system of $G$ we get that
$\CAY(S,C'')$ is triangle-free. Hence it has at most $2(mk-2)=6m-4$ edges by Lemma~\ref{lem:Euler} -- a contradiction.

If (up to relabeling of $R_k$) we have
$C'':=\{(a_1,r_1),(a_2,r_2),(a_3,r_2)\}\subseteq C$ then we consider
the subgraph $\CAY(G\times R_2,C'')$ of $\CAY(S,C)$. By
Lemma~\ref{lem:count} we know that $\CAY(G\times R_2,C'')$ has
$4.5m$ edges, where $m=|G|$. As in the previous case $\CAY(G\times
R_2,C'')$ is triangle-free and has most $2(mk-2)=4m-4$ edges by
Lemma~\ref{lem:Euler} -- a contradiction.

If (up to relabeling of $R_k$) $C'':=\{(a_1,r_1),(a_2,r_1),(a_3,r_1),(x,r_2)\}\subseteq C$ again we consider the subgraph
$\CAY(G\times R_2,C'')$ of $\CAY(S,C)$. But now we have to distinguish two subcases:

If $x\neq e$ then by the lower bound in~\eqref{form:count} we know that
$\CAY(G\times R_2,C'')$ has at least $6m$ edges, where $m=|G|$.
On the other hand Lemma~\ref{lem:Euler} gives an upper bound of $6m-6$ -- a contradiction.

If $x=e$, then $\CAY(G\times R_2,C'')$ has $5.5m$ edges and we have to come up with a stronger upper bound than
Lemma~\ref{lem:Euler}
for this particular case. Note that in $\CAY(G\times R_2,C'')$ every edge may appear in two triangles except for
edges of the form
$\{(g,r_2),(ga_i,r_1)\}$ for $i=1,2,3$.
The latter edges appear only in the triangle $\{(g,r_2),(ga_i,r_1),(g,r_1)\}$ and there are $3m$ of them.
We therefore have that the number of triangular faces
$|F_3|$ is bounded from above by $\frac{2|E|-3m}{3}$ and there are at least $\frac{3m}{4}$ larger faces.
Plugging this into Euler's Formula yields $|E|\leq \frac{21}{4}m-6$, which is less than $5.5m$ -- a contradiction.
\end{proof}

We now turn to the remaining cases. Here, edge counting does not
suffice for proving non-planarity. Instead we will use Wagner's
Theorem, i.e., we will find $K_5$ and $K_{3,3}$ minors to prove
non-planarity. (Here, $K_{m,n}$ denotes the complete bipartite graph
with partition sets of size $m$ and $n$, respectively.) First we
prove a lemma somewhat complementary to Lemma~\ref{lem:local}.

\begin{lem}\label{lem:local2}
 Let $G$ be a group with generating system $C'=\{a,b\}$ where $a$ is of order two and $b$ of larger order such
 that the neighborhood
 of some $a$-edge in $\Cay(G,C')$ looks as depicted in Figure~\ref{fig:local2}.
 Then $\CAY(G\times R_k,C)$ is non-planar
 if $C'\subseteq \pi_G(C)$ and $k\geq 2$.
\end{lem}

\begin{figure}[ht]
\begin{center}
\includegraphics[width = .4\textwidth]{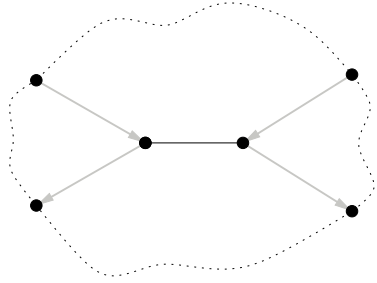}
\caption{The black edge is the $a$-edge, the gray arcs are $b$-arcs, the dotted curves correspond to paths, vertex-disjoint from all other elements of the figure.}\label{fig:local2}
\end{center}
\end{figure}

\begin{proof}
We distinguish two cases of what $C$ looks like:

 If (up to relabeling of $R_k$) we have $C''=\{(a,r_1),(b,r_1)\}\subseteq C$ consider $\Cay(G\times R_2,C'')$.
 This graph contains $\Cay(G\times R_1,C'')\cong \Cay(G,C')$. Consider the $a$-edge from Figure~\ref{fig:local2}, say
 that it connects vertices $(e,r_1)$ (left) and $(a,r_1)$ (right). Add vertices $(e,r_2)$ and $(a,r_2)$ to the picture.
 The first has an edge to $(a,r_1)$ and the bottom-left vertex $(b,r_1)$. The second has an edge to $(e,r_1)$ and the
 bottom-right vertex $(ab,r_1)$. Contracting these two $2$-paths to a single edge each, as well as the left, bottom and
 right dotted path to single edges and the top dotted path to a single vertex we obtain $K_5$, see Figure~\ref{fig:K5}.

  \begin{figure}[ht]
\begin{center}

\psfrag{a1}{$(a,r_1)$}
\psfrag{a2}{$(a,r_2)$}
\psfrag{e1}{$(e,r_1)$}
\psfrag{e2}{$(e,r_2)$}
\psfrag{b1}{$(b,r_1)$}
\psfrag{ab1}{$(ab,r_1)$}
\psfrag{b-11}{$(b^{-1},r_1)$}
\psfrag{b-1a1}{$(b^{-1}a,r_1)$}
\psfrag{b2}{$(b,r_2)$}

\includegraphics[width = .4\textwidth]{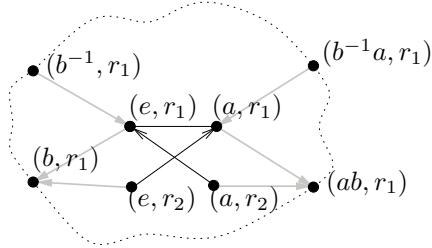}
\caption{Finding $K_5$.}\label{fig:K5}
\end{center}
\end{figure}

 If (up to relabeling of $R_k$) we have $C''=\{(a,r_1),(b,r_2)\}\subseteq C$ again consider $\Cay(G\times R_2,C'')$.
 Denote by $H\subseteq G$ the elements corresponding to the vertices of Figure~\ref{fig:local2} without the two vertical dotted paths. Take the $b$-arcs in $\Cay(H\times\{r_2\},C'')$ and all $a$-edges $\Cay(H\times R_2,C'')$. As in the first case we assume without loss of generality that the central $a$-edge connects from left to right elements $e$ and $a$ in $\Cay(G,C'')$. This edge will now be represented by a $3$-path $(e,r_2),(a,r_1),(e,r_1),(a,r_2)$. Also in the dotted paths we need to replace $a$-edges by some new paths. We do this in the same way as with the central edge. The paths resulting this way from the dotted paths will again be pairwise disjoint and we draw them dotted in Figure~\ref{fig:K33}.
  
 To obtain a $K_{3,3}$-minor focus on the $3$-path $(e,r_2),(a,r_1),(e,r_1),(a,r_2)$ representing the central $a$-edge in our argument. We include the $b$-arcs $((e,r_1),(b,r_2))$ and $((a,r_1),(ab,r_2))$ starting from the inner vertices of this $3$-path.  

 \begin{figure}[ht]
\begin{center}

\psfrag{a1}{$(a,r_1)$}
\psfrag{a2}{$(a,r_2)$}
\psfrag{e1}{$(e,r_1)$}
\psfrag{e2}{$(e,r_2)$}
\psfrag{b1}{$(b,r_1)$}
\psfrag{ab1}{$(ab,r_1)$}
\psfrag{b-11}{$(b^{-1},r_1)$}
\psfrag{b-1a1}{$(b^{-1}a,r_1)$}
\psfrag{b2}{$(b,r_2)$}

\includegraphics[width = .4\textwidth]{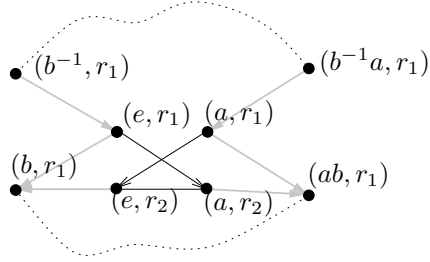}
\caption{Finding $K_{3,3}$.}\label{fig:K33}
\end{center}
\end{figure}

 Contract the dotted path between $(b,r_2)$ and $(ab,r_2)$ and the one between $(b^{-1},r_2)$  and $(ab^{-1},r_2)$ to a single edge, respectively. Last, we contract the two arcs $((b^{-1},r_2), (e,r_2))$ and $((ab^{-1},r_2),(a,r_2))$. The resulting graph is $K_{3,3}$.
\end{proof}

The lemma yields:

\begin{prop}\label{prop:nonplanar2}
 The right groups $\Z_2\times H\times
R_k$ with $H$ any of $\mathbb{Z}_{2n},A_4$ where $n\geq 2$ and $k=2,3$ are
not planar.
\end{prop}
\begin{proof}
 Let $S=G\times R_k$ be one of the right groups from the statement and suppose that $C$ is a generating system of $S$
 such that $\CAY(S,C)$ is planar. By Theorem~\ref{thm:project} we know that there is a generating system
$C'\subseteq \pi_G(C)$ of $G$ such that $\CAY(G,C')$ is planar. Comparing with Table~\ref{table} we see that for $G\in\{\Z_2\times\Z_{2n},\Z_2\times A_4\}$ there is exactly one planar generating system. The preconditions of Lemma~\ref{lem:local2} are
satisfied for $G=Z_2\times\mathbb{Z}_{2n}$, which is easy to check directly and for
$G=\Z_2\times A_4$ we refer to Figure~\ref{fig:A42}. Thus, the respective Cayley graphs cannot be planar.
\end{proof}

Now we have proved:
 \begin{thm}\label{thm:nonplanar}
  The right groups $\Z_2\times H\times R_k$ with $H\in
 \{\Z_{2(n+1)},D_{2n},A_4,S_4,A_5\}$, $n\geq 1$, and $k\geq 2$ are not planar.
\end{thm}

\section{Conclusions}\label{sec:con}
From the previous results (Theorem~\ref{thm:positiv} and Theorem~\ref{thm:nonplanar}) we get our main theorem:
\begin{thm}\label{thm:charact} The planar right groups $G\times R_k$ with $k\geq 2$ are exactly of the form
$G\in\{\{e\},\mathbb{Z}_n, D_n, S_4, A_4, A_5\}$ and $k\leq 3$ and $\{e\}\times R_4$. Here $\{e\}$ denotes
the one-element group.
\end{thm}

\begin{rem}
The non-planarity proof for the right groups in
Section~\ref{sec:nonplanar} is slightly longer and more involved
than an alternative proof entirely using minors. We hope that the
present proof is generalizable to higher and maybe even
non-orientable genus.
\end{rem}

Before going into further questions let us repeat the conjecture made in this paper. Recall that for $C\subseteq S=G\times R_k$ we denote by $\pi_G(C):=\{g\in G\mid (g,r_i)\in C$ for some $i\in[k]\}$.
\begin{conj*}
 Let $S=G\times R_k$ and $C$ a generating set. Then $\CAY(G,\pi_G(C))$ is a minor of $\CAY(S,C)$.
\end{conj*}
Answering this conjecture in the affirmative would be particularly interesting for relating the genus of right groups to the genus of groups, as done in the present paper for the planar case.

We conclude with a list of problems, starting pretty close to this paper's results but then going further into the exciting area of Cayley graphs of semigroups and their topology. 
%
%

\begin{quest}
We think the genus of the non-planar right groups from Theorem
\ref{thm:nonplanar} is rather large. Which genus do they have?
\end{quest}

\begin{quest}
It is easy to see, that the right groups $\Z_n \times R_2$ and $R_2,
R_3$ are exactly the outerplanar (non-group) right groups. Can we characterize outerplanar semigroups? Solomatin has obtained first results
in~\cite{Sol-11}. His paper in fact contains and relies on many results on planar
finite free commutative semigroups, which have been proved in other
 not easily accessible places.
\end{quest}

\begin{quest}
 The planar groups are exactly the discrete subgroups of isometries of the sphere $O(3)$. The planar Cayley graphs of such a group $G$ arise in the following way: Chose a region whose translates with respect to $G$ tessellate the sphere. The Cayley graph is the dual graph of the resulting tessellation. Is there an analogue construction for planar semigroups?
\end{quest}

\begin{quest}The graphs of all Archimedean and Platonic solids are Cayley graphs
of a group with minimal generating system with three exceptions: the
Dodecahedron, the Icosidodecahedron, and the antiprisms (in
particular the tetrahedron). The antiprism is the Cayley graph of a
group with non-minimal generating system though, e.g.,
$\CAY(\mathbb{Z}_{2n},\{1,2\})$. The other two are not even such. Are
they underlying graphs of directed Cayley graphs of semigroups (with
minimal generating system)? It has been shown in~\cite{Hao-11} that the Dodecahedron graph is an induced subgraph of a Brandt semigroup Cayley graph.
\end{quest}

\begin{quest}
As above,  a natural questions arising in this type of study is, whether a given graph is the graph of a semigroup. Find Sabidussi's Theorem~\cite{Sab-58} for semigroups, i.e., 
an abstract characterization of Cayley graphs of semigroups.
It is clear that the Cayley graph of a semigroup $S$ has a homomorphic image $S'$ of $S$ as a
subsemigroup of its endomorphism monoid. What else has to be asked
for to make this a sufficient condition? There has been some work
into that direction (\cite{Kel-02,Kel-03,Kel-06,Sol-11} and Chapter
11.3 of~\cite{Kna-11}).
\end{quest}

\begin{quest}
 Another pretty different but also natural notion of planarity for a (semi)group is to require that the Hasse diagram of its lattice of sub(semi)groups be planar. Groups that are planar in this sense have been completely classified recently~\cite{Boh-06}. How about semigroups? 
\end{quest}

\subsubsection*{Acknowledgements.}
The exposition of the paper greatly benefited from the valuable comments of an anonymous referee.

\bibliography{lit}
\bibliographystyle{amsplain}

\end{document}